\documentclass[a4paper,11pt]{article}

\usepackage{amsmath}
\usepackage{amsfonts}
\usepackage{amssymb}
\usepackage{amstext    }
\usepackage{amsthm    }

\usepackage[a4paper]{geometry}
\usepackage{mathrsfs}

\usepackage{hyperref}

\usepackage{color}

\usepackage[T1]{fontenc}
\usepackage{graphicx}
\usepackage[all,cmtip]{xy}

\usepackage{epsf}
\usepackage{psfrag}

\usepackage{fancyhdr}

\usepackage{url}

\newtheorem{theorem}{Theorem}[section]
\newtheorem{lemma}[theorem]{Lemma}

\newtheorem{proposition}[theorem]{Proposition}
\newtheorem{remark}[theorem]{Remark}
\newtheorem{corollary}[theorem]{Corollary}

\newcommand{\cali}[1]{\mathscr{#1}}

\numberwithin{equation}{section}

\newcommand{\ddc}{{dd^c}}

\newcommand{\Jac}{{\rm Jac}}
\newcommand{\supp}{{\rm supp}}

\newcommand{\bif}{{\rm Bif}}
\newcommand{\proj}{{\rm proj}}
\newcommand{\per}{{\rm Per}}

\newcommand{\crit}{{\rm Crit}}

\newcommand{\Jc}{\cali{J}}

\newcommand{\Pb}{\mathbb{P}}
\newcommand{\Cb}{\mathbb{C}}

\newcommand{\Zb}{\mathbb{Z}}

\newcommand{\Rb}{\mathbb{R}}

\usepackage[T1]{fontenc}
\usepackage[english]{babel}

\title{Dynamics of fibered endomorphisms of $\mathbb P ^k$}
\author{Christophe Dupont and Johan Taflin\footnote{Research  partially supported by ANR project Fatou ANR-17-CE40-0002-01.}}
\begin{document}

\maketitle

\begin{abstract}
We study the structure and the Lyapunov exponents of the equilibrium measure of endomorphisms of $\Pb^k$ preserving a fibration. We extend the decomposition of the equilibrium measure obtained by Jonsson for polynomial skew products of $\Cb^2$. We also show that the sum of the sectional exponents satisfies a Bedford-Jonsson formula when the fibration is linear, and that this function is plurisubharmonic on families of fibered endomorphisms. In particular, the sectional part of the bifurcation current is a closed positive current on the parameter space.
\end{abstract}



\section{Introduction}
Let $f$ be a holomorphic endomorphism of $\Pb^k$ of degree $d\geq2$ which preserves a rational fibration parametrized by a projective space i.e. there exist a dominant rational map $\pi\colon \Pb^k\dashrightarrow\Pb^r$ and a holomorphic map $\theta\colon\Pb^r\to\Pb^r$ such that
\begin{equation}\label{eq-com}
\pi\circ f=\theta\circ \pi.
\end{equation}
The generic fiber of $\pi$ has dimension $q:=k-r.$ Another way to express \eqref{eq-com} is that $f$ permutes the fibers of $\pi$ and this permutation is given by $\theta.$ We are interested in the relationships between the dynamics of $f$ and the one of $\theta.$

This type of maps has been recently used to exhibit interesting dynamical phenomena in $\Pb^2$ (see \cite{dujardin-non-laminar}, \cite{wandering}, \cite{bianchi-t-desboves}, \cite{dujardin-bif}, \cite{t-blender}). All these examples, except \cite{bianchi-t-desboves}, come from polynomial skew products of $\Cb^2$, whose dynamical properties have been studied by Jonsson in \cite{jonsson-skew}. It is therefore interesting for future examples to extend the results of \cite{jonsson-skew} to a broader framework. Our initial motivation was to study the particular case where $\pi$ is the standard linear fibration defined by $\pi[y:z]=[y]$ with $y:=(y_0,\ldots,y_r)\in\Cb^{r+1}$ and $z=(z_0,\ldots,z_{q-1})\in\Cb^q.$ 
 However, some of the techniques can be extended to a more general setting. In what follows, we choose the setting of each result in order to avoid unnecessary technical details. We refer to the end of this introduction for the possible scope of the techniques, in particular when $k=2$ thanks to the works of Dabija-Jonsson (\cite{dabija-jonsson-pencil}, \cite{dabija-jonsson-web}) and Favre-Peirera (\cite{favre-pereira-foliation}, \cite{favre-pereira-web}) on endomorphisms of $\Pb^2$ preserving a fibration, a foliation or a web.

\paragraph{Green currents and Lyapunov exponents --}The maps $f,$ $\pi$ and $\theta$ are given by homogeneous polynomials and a simple computation shows that $f$ and $\theta$ have the same degree. Both maps have a Green current, $T_f$ and $T_\theta$ respectively, which are positive closed $(1,1)$-currents with continuous local potentials. Their self-intersections are well-defined and their supports define dynamically meaningful filtrations
$$\Jc_i(f):=\supp(T_f^i)\ \ \ \text{ and }\ \ \ \Jc_j(\theta):=\supp(T_\theta^j),$$
for $i\in\{1,\ldots,k\}$ and $j\in\{1,\ldots,r\}.$ The \textit{equilibrium measures} of $f$ and $\theta$ are defined by $\mu_f:=T^k_f$ and $\mu_\theta:=T^r_\theta.$ Since $f$ and $\theta$ are semi-conjugated by $\pi,$ a natural question is whether there exists a relation between $T_f^i$ and the pull back of $T_\theta$ by $\pi.$
Our first result gives such a relationship if $i>q.$ More precisely, we will see in Section \ref{sec-basics} how to define $\pi^*T_\theta$ and if $S$ denotes the result normalized by its mass,
$$S:=\frac{\pi^*T_\theta}{\|\pi^*T_\theta\|},$$
then we have the following result.
\begin{theorem}\label{th-mu}
Let $f\colon\Pb^k\to\Pb^k$  and $\theta\colon\Pb^r\to\Pb^r$ be two endomorphisms of degree $d\geq2.$ Assume there exists a dominant rational map $\pi\colon\Pb^k\dashrightarrow\Pb^r$ whose indeterminacy set $I(\pi)$ is disjoint from $\Jc_q(f)$ and such that $\theta\circ\pi=\pi\circ f.$ Then for $j\in\{1,\ldots, r\},$ the current $S^j$ is well-defined, satisfies $S^j\neq T_f^j$ and $T_f^{q+j}=T_f^q\wedge S^j.$ In particular, $\mu_f=T_f^q\wedge S^r$ and $\pi_*\mu_f=\mu_\theta.$
\end{theorem}
Let us emphasize that the proof only relies on the properties of the currents $T_f$ and $T_\theta$ and is coordinate free. Using the classification obtained in \cite{dabija-jonsson-pencil}, one can check easily that the assumption $\Jc_q(f)\cap I(\pi)=\varnothing$ is always satisfied when $k=2,$ in which case $q=1$. This is also the case for the standard linear fibration in any dimension (see Lemma \ref{le-supp-sans-para}). In general, we know no example where this assumption does not hold.

The main point in Theorem \ref{th-mu} is the formula $\mu_f=T_f^q\wedge S^r$ which can be seen as a generalization of the decomposition of $\mu_f$ obtained by Jonsson \cite{jonsson-skew} for polynomial skew products of $\Cb^2.$ Indeed, for $\mu_\theta$-almost every $a\in\Pb^r$ the fiber $L_a:=\overline{\pi^{-1}(a)}$ has dimension $q$ and we can define the probability measure
$$\mu_a:=\frac{T_f^q\wedge[L_a]}{\|\pi^*T_\theta\|^r}.$$

\begin{corollary}\label{cor-decomp}
Let $\phi\colon\Pb^k\to\Rb$ be a continuous function. Under the assumptions of Theorem \ref{th-mu} we have
$$\int_{\Pb^k}\phi(x)d\mu_f(x)=\int_{\Pb^r}\left(\int_{L_a}\phi(x)d\mu_a(x)\right)d\mu_\theta(a).$$
\end{corollary}

The other results in this paper can also be seen as consequences of the formula  $\mu_f=T_f^q\wedge S^r$ and the main technical difficulties come from the fact that the currents $S$ and $[L_a]$ are singular. As a direct consequence of Theorem \ref{th-mu}, we obtain in the following result that if $\mu_\theta$ is absolutely continuous with respect to Lebesgue measure (i.e. $\theta$ is a Latt\`es mapping of $\Pb^r$, see \cite{dupont-lattes}) then $\mu_f$ is absolutely continuous with respect to the trace measure $\sigma_{T_f^q}:=T_f^q\wedge\omega_{\Pb^k}^r.$ Here, $\omega_{\Pb^k}$ (resp. $\omega_{\Pb^r}$) is the Fubini-Study form on $\Pb^k$ (resp. $\Pb^r$)  normalized such that $\omega_{\Pb^k}^k$ (resp. $\omega_{\Pb^r}^r$) is a probability measure.
\begin{corollary}\label{cor-abso}
Under the assumptions of Theorem \ref{th-mu}, if $\mu_\theta<<\omega^r_{\Pb^r}$ then $\mu_f<<\sigma_{T_f^q}.$
\end{corollary}

That applies for Desboves mappings of $\Pb^2$, studied in \cite[Section 4]{bdm-elliptic} and \cite{bianchi-t-desboves}, they indeed induce  a Latt\`es mapping on a pencil of lines. Let us note that when $k=2$, the property $\mu_f<<\sigma_{T_f}$ implies that the smallest exponent of $\mu_f$ is minimal, equal to ${1 \over 2} \log d$, see \cite[Theorem 3.6]{Dujardin-Fatou}. In particular the Lyapunov exponents of Desboves mappings are $\lambda_1 > \lambda_2 = {1\over 2} \log d$, with $d=4$. The following Theorem generalizes that semi-extremal property to fibered endomorphisms satisfying $\mu_\theta<<\omega^r_{\Pb^r}$. It is a consequence of $\pi_*\mu_f=\mu_\theta$ and holds for more general smooth dynamical systems.

\begin{theorem}\label{th-lyap}
Let $f,$ $\pi$ and $\theta$ be as in Theorem \ref{th-mu}. If $\Lambda$ is a Lyapunov exponent of multiplicity $m$ for  $\mu_\theta$ then $\Lambda$ is a Lyapunov exponent of multiplicity at least $m$ of $\mu_f.$
\end{theorem}

\paragraph{Standard linear fibration --} For the next results, we restrict ourselves to the cases where $\pi$ is the standard linear fibration where the result below is already new when $k=2$ and $r=1.$ The indeterminacy set $I(\pi)$ of $\pi$ corresponds to $\{y=0\}\simeq \Pb^{q-1}$ and each fiber $L_a=\overline{\pi^{-1}(a)}$ is a linear projective space $\Pb^q$ in which $I(\pi)$ can be identified with the hyperplane at infinity, i.e. $L_a\setminus I(\pi)\simeq \Cb^q.$ If $f$ preserves the fibration defined by $\pi$ then $f$ acts on each periodic fibers as a regular polynomial endomorphism of $\Cb^q.$ This class of maps have been studied by Bedford-Jonsson in \cite{bedford-jonsson}. In particular, they obtained a formula for the sum of the Lyapunov exponents of the equilibrium measure. More precisely, let $R$ be a regular polynomial endomorphism of $\Cb^q$ of degree $d$ i.e. $R$ extends to an endomorphism of $\Pb^q$ of degree $d.$  We denote by $T_R,$ $\crit_R$ and $G_R$ respectively the Green current, the critical set and the Green function in $\Cb^q$ of $R.$ The restriction of $R$ to the hyperplane at infinity $\Pb^q\setminus\Cb^q\simeq\Pb^{q-1}$ is an endomorphism of $\Pb^{q-1}$ and we denote by $\Lambda_0$ the sum of the Lyapunov exponents of its equilibrium measure.
\begin{theorem}[Bedford-Jonsson \cite{bedford-jonsson}]\label{bdjn}
Let $R$ be a regular polynomial endomorphism of $\Cb^q$ of degree $d.$ The sum $\Lambda_R$ of the Lyapunov exponents of its equilibrium measure satisfies
$$\Lambda_R=\log d+\Lambda_0+\langle T_R^{q-1}\wedge[\crit_R],G_R\rangle.$$
\end{theorem}
We give a generalization of this formula in the fibered setting. To this aim, we introduce some notations. If $\pi\circ f=\theta\circ\pi$ then we denote by $\Lambda_f$ (resp. $\Lambda_\theta$) the sum of the Lyapunov exponents of $\mu_f$ (resp. $\mu_\theta$). Theorem \ref{th-lyap} implies that $\Lambda_f=\Lambda_\theta+\Lambda_\sigma$ where $\Lambda_\sigma$ is the sum of the Lyapunov exponents of $\mu_f$ in the direction of the fibers. The indeterminacy set $I(\pi)\simeq\Pb^{q-1}$ is invariant by $f$ thus $f_{I(\pi)}$ can be seen as an endomorphism of $\Pb^{q-1}$ and we denote by $\Lambda_0$ the sum of the Lyapunov exponents of this restriction.

Since $f$ preserves the fibration, the set $\crit_f$ is not irreducible. Some irreducible components of $\crit_f$ are foliated by fibers of $\pi$ and constitute the ``fibered'' part of $\crit_f.$ The remaining part is its ``sectional'' part. Indeed, as the standard linear fibration $\pi$ is a submersion outside $I(\pi),$ we have a decomposition in terms of currents $[\crit_f]=[C_\infty]+[C_\sigma]$ where $[C_\infty]:=\pi^*[\crit_\theta]$ and $[C_\sigma]$ is the current of integration on the sectional part of $\crit_f$ (see Lemma \ref{le-crit}).

Finally, we define the \textit{relative Green function} as the unique lower semicontinuous function $G\colon\Pb^k\to[0,+\infty]$ such that $\ddc G=T_f-S$ and $\min G=0.$
\begin{theorem}\label{th-bj}
Let $f$ be an endomorphism of $\Pb^k$ of degree $d\geq2$ which preserves the standard linear fibration. Then
$$\Lambda_\sigma=\log d+\Lambda_0+\langle T_f^{q-1}\wedge S^r\wedge[C_\sigma],G\rangle.$$
In particular, $\Lambda_\sigma\geq \frac{q+1}{2}\log d.$
\end{theorem}
The idea of the proof is to apply the formula of Bedford-Jonsson to each $n$-periodic fiber of $f.$ Since the current $S^r$ can be seen as the limit of the average of the currents of integration on the $n$-periodic fibers, we obtain the above formula at the limit. However, in order to implement that idea, we need some additional care since the currents involved are singular.

\paragraph{Families of fibered endomorphisms --}We now consider a family $(f_\lambda)_{\lambda\in M}$ of endomorphisms of $\Pb^k$ which preserves the standard linear fibration $\pi$ i.e. there exists a family $(\theta_\lambda)_{\lambda\in M}$ of endomorphisms of $\Pb^r$ such that $\pi\circ f_\lambda=\theta_\lambda\circ\pi.$ The family $(f_\lambda)_{\lambda\in M}$ induces a dynamical system $f(\lambda, x):=(\lambda,f_\lambda(x))$ on $M\times\Pb^k$ whose critical set $\crit_f$ is the gluing of the critical sets of $f_\lambda,$ $\lambda\in M.$ In the same way, there exists a positive closed $(i,i)$-current $T^i_f$  on $M\times\Pb^k$ whose slices are equal to $T^i_{f_\lambda}.$ In \cite{bassanelli-berteloot-bif} Bassanelli and Berteloot established a formula between the currents $[\crit_f]$ and $T^k_f$ and the $\ddc$ of $\Lambda_f\colon\lambda\mapsto\Lambda_{f_\lambda}$ where $\Lambda_{f_\lambda}$ is the sum of the Lyapunov exponents of $\mu_{f_\lambda}$ (see also \cite{pham}). They proved that
$$\ddc\Lambda_f=p_*(T^k_f\wedge[\crit_f]),$$
where $p\colon M\times\Pb^k\to M$ is the projection. This current is called the \textit{bifurcation current} $T_\bif(f)$ and its support coincides with several bifurcation phenomena in the family $(f_\lambda)_{\lambda\in M}$ (see \cite{bbd-bif}).

Similar objects can be defined for the family $(\theta_\lambda)_{\lambda\in M}$ and again, it is natural to inquire into their interplays with the ones defined for $(f_\lambda)_{\lambda\in M} .$ Since this family preserves the fibration, as above the set $\crit_f$ is not irreducible and we have $[\crit_f]=[C_\infty]+[C_\sigma]$ where $[C_\infty]:=\Pi^*[\crit_\theta]$ with $\Pi(\lambda,x)=(\lambda,\pi(x)).$ This decomposition induces a decomposition of $T_\bif(f)$ in a fibered part and a sectional part. The result below states that the fibered part of $T_\bif(f)$ coincides with $T_\bif(\theta).$ 
\begin{theorem}\label{th-pham}
Let $M$ be a complex manifold and consider two holomorphic families $(f_\lambda)_{\lambda\in M}$ and $(\theta_\lambda)_{\lambda\in M}$ of endomorphisms of $\Pb^k$ and $\Pb^r$ respectively such that $\pi\circ f_\lambda=\theta_\lambda\circ\pi$ where $\pi$ is the standard linear fibration. The $(1,1)$-current $T_{\bif,\sigma}(f):=T_\bif(f)-T_\bif(\theta)$ is positive. Moreover, if $S^r:=\Pi^*T_\theta^r$ then
$$T_\bif(\theta)=p_*(T_f^q\wedge S^r\wedge[C_\infty]),\ \ T_{\bif,\sigma}(f)=p_*(T_f^q\wedge S^r\wedge[C_\sigma]).$$
\end{theorem}
A different way of seeing this result is the following. By Theorem \ref{th-lyap}, we know that $\Lambda_\sigma=\Lambda_f-\Lambda_\theta$ is the sum of the Lyapunov exponents of $\mu_f$ which are not in the Lyapunov spectrum of $\mu_\theta.$ Theorem \ref{th-pham} yields that $\Lambda_\sigma$ is a plurisubharmonic function on $M$ and gives a formula of $\ddc\Lambda_\sigma$ in terms of currents on $M\times\Pb^k,$
$$\ddc\Lambda_\sigma=p_*(T_f^q\wedge S^r\wedge[C_\sigma]).$$
Notice that Astorg-Bianchi initiated in \cite{astorg-bianchi-skew} the study of bifurcations for skew products of $\Cb^2$ and proved in that setting that $\Lambda_\sigma$ is plurisubharmonic.

Following an idea coming from \cite{astorg-bianchi-skew}, for each $n\geq1$ we can consider the bifurcation current associated to the dynamics of the family $(f_\lambda)_{\lambda\in M}$ on the $n$-periodic fibers. To be more precise, if $\lambda\in M$ and $a\in\Pb^r$ are such that $\theta_\lambda(a)=a$ then we denote by $\Lambda(f^n_{\lambda|L_a})$ the sum of Lyapunov exponents of $f^n_{\lambda|L_a}$ seen as a polynomial endomorphism of $L_a\simeq\Pb^q.$ Then we define
$$\Lambda_{\sigma,n}(\lambda):=\frac{1}{nd^{rn}}\sum_{\theta_\lambda^n(a)=a}\Lambda(f^n_{\lambda|L_a})\ \ \text{ and }\ \ T_{\bif,n}(f):=\ddc\Lambda_{\sigma,n}.$$
It is easy to see that if all the cycles of the family $(\theta_\lambda)_{\lambda\in M}$ can be followed holomorphically on $M$ then $\Lambda_{\sigma,n}$ is plurisubharmonic. Actually, this holds in general and we can express the current $T_{\bif,\sigma}(f)$ in terms $T_{\bif,n}(f).$
\begin{corollary}\label{cor-bif}
Let $(f_\lambda)_{\lambda\in M},$ $(\theta_\lambda)_{\lambda_\in M}$ and $\pi$ be as in Theorem \ref{th-pham}. For each $n\geq1$ the function $\Lambda_{\sigma,n}$ is plurisubharmonic and
$$T_{\bif,\sigma}(f)=\lim_{n\to\infty}T_{\bif,n}(f).$$
Moreover, if $[\per_{\theta,n}]$ denotes the current of integration on $\{(\lambda,a)\in M\times\Pb^r\,|\, \theta^n_\lambda(a)=a\}$ by taking into account multiplicities, then
$$T_{\bif,n}=\frac{p_*(T_f^q\wedge \Pi^*[\per_{\theta,n}]\wedge[C_\sigma])}{d^{rn}}.$$
\end{corollary}
The first part of this result was obtained in \cite[Corollary 4.8]{astorg-bianchi-skew} in the special case of skew products of $\Cb^2$ under the hypothesis that $\ddc\Lambda_\theta=0.$ And, as observed by Astorg-Bianchi, a consequence of Corollary \ref{cor-bif} is that if the dynamics on $(f_\lambda)_{\lambda\in M}$ bifurcates on one periodic fiber then, asymptotically when $n\to\infty,$ it bifurcates on a positive proportion of the $n$-periodic fibers.

\paragraph{Final remarks and outline of the paper --}To conclude this introduction, let us explain in which setting results similar to Theorem \ref{th-mu} and Theorem \ref{th-lyap} could be obtained. First, observe that the assumption that the base space is $\Pb^r$ is unnecessary as long as $\dim(I(\pi))\leq q-1.$ Indeed, if $\pi$ is a dominant meromorphic map between $\Pb^k$ and a compact complex manifold $X$ of dimension $r$ with $\dim(I(\pi))\leq q-1,$ ($q=k-r$), then the restriction of $\pi$ to a generic linear subspace of dimension $r$ in $\Pb^k$ gives a surjective holomorphic map from $\Pb^r$ to $X.$ Then by results in \cite[Section 2 \& 3]{demailly-hwang-peternell}, $X$ is projective and then by \cite{lazarsfeld-some}, $X$ is isomorphic to $\Pb^r.$ Observe that this argument uses the smoothness of the base. The case with a singular base might appear naturally but goes beyond the scope of this paper.

Another natural setting is the following. Assume that $f$ is an endomorphism of $\Pb^k$ which preserves a family $(L_a)_{a\in X}$ of algebraic sets of dimension $q$ and of degree $\alpha$ parametrized by a complex manifold $X$ of dimension $r,$ i.e. there exists an endomorphism $\theta$ of $X$ such that $f(L_a)=L_{\theta(a)}.$ It is natural to expect that under some assumptions on the family $(L_a)_{a\in X}$ and if $\theta$ possesses an equilibrium measure $\mu_\theta,$ the measure $\mu_f$ can be written as $\mu_f=T^q_f\wedge S^r$ where
$$S^r:=\int_{X}\frac{[L_a]}{\alpha}d\mu_\theta(a).$$
Indeed, it is easy to check, using the classifications in \cite{dabija-jonsson-web} and \cite{favre-pereira-web} and the proof of Theorem \ref{th-mu}, that this is the case when $k=2$ and the family $(L_a)_{a\in X}$ defines a web with algebraic leaves. However, in some of these examples $S=T_f.$ This holds for (ii)-(iv) in \cite[Theorem A]{dabija-jonsson-web} and for (ii) in \cite[Theorem E]{favre-pereira-web}. Otherwise, $S\neq T_f$ in these theorems.

Finally, let us mention that results of Dinh-Nguy\^en-Truong \cite{dinh-nguyen-truong-fibration}, \cite{dinh-nguyen-truong-domi} suggest that one might expect some of the results above (as $\pi_*\mu_f=\mu_\theta$ and Theorem \ref{th-lyap}) to extend to the case of dominant meromorphic self-maps of compact K\"{a}hler manifolds with large (or dominant) topological degree which preserve meromorphic fibrations.

The paper is organized as follows. In Section \ref{sec-basics}, we give some technical results on the pull-back and the intersection of currents. In Section \ref{sec-struc} and Section \ref{sec-lyap}, we prove Theorem \ref{th-mu}, Corollary \ref{cor-abso} and Theorem \ref{th-lyap} respectively. In Section \ref{sec-bj}, we restrict ourselves to the cases where $\pi$ is the standard linear fibration and we establish our generalized Bedford-Jonsson formula in that context. Section \ref{sec-bif} is devoted to bifurcations of such maps.

\paragraph{Acknowledgements --} The authors would like to thank Charles Favre for useful comments on the preliminary version of this paper.

\section{Basics on pluripotential theory}\label{sec-basics}
In Section \ref{sec-bj} and Section \ref{sec-bif}, we will intersect currents supported by fibers of a rational map $\pi\colon\Pb^k\dashrightarrow\Pb^r$ and singular currents or functions. Moreover, we will need that the result depends continuously on the fiber. It seems complicated to obtain these statements for an arbitrary dominant rational map $\pi.$ The aim of the first part of this section is to give two results in that direction when $\pi$ is the standard linear fibration. In a second part, we explain how to define the pull-back $\pi^*\tau$ of a positive closed $(1,1)$-current $\tau$ on $\Pb^r$ by a rational map and how to define its self-intersections $(\pi^*\tau)^j$ in the setting of Theorem \ref{th-mu}.

\subsection{Continuous families of currents}
Let $\pi\colon\Pb^k\dashrightarrow\Pb^r$ be the standard linear fibration defined by $\pi[y:z]=[y]$ where $y:=(y_0,\ldots,y_r)\in\Cb^{r+1}$ and $z=(z_0,\ldots,z_{q-1})\in\Cb^q.$ We recall that the indeterminacy set $I(\pi)$ of $\pi$ corresponds to $\{y=0\}\simeq \Pb^{q-1}$ and each fiber $L_a:=\overline{\pi^{-1}(a)}$ is a projective space in which $I(\pi)$ can be identified with the hyperplane at infinity, i.e. $L_a\setminus I(\pi)\simeq \Cb^q.$

In the following two results, we consider an integer $0\leq l\leq k$ and a family $(R_a)_{a\in\Pb^r}$ of positive closed $(k-l,k-l)$-currents in $\Pb^k$ such that $a\mapsto R_a$ is continuous. We also consider an open set $U\subset\Pb^k$ and an upper semicontinuous function $v\colon U\to[-\infty,0]$ such that $\ddc v=T_1-T_2,$ where $T_1$ and $T_2$ are two positive closed $(1,1)$-currents where $T_2$ has continuous local potentials.
\begin{lemma}\label{le-conti}
Assume there exist two analytic subsets $X,Y\subset\Pb^k$ such that $v$ is continuous on $U\setminus X$ and such that for all $a\in\Pb^r$ we have $\supp(R_a)\subset L_a\cap Y$ and $\dim(L_a\cap X\cap Y)\leq l-1.$ Then, for all $a\in\Pb^r$ the current $vR_a$ is well-defined on $U$ and depends continuously on $a.$
\end{lemma}
\begin{proof}
The facts that $vR_a$ is well-defined and that its mass is locally uniformly bounded with respect to $a$ follow easily from the Oka inequality obtained by Forn{\ae}ss-Sibony \cite{fs-oka}. In order to give some details, we freely use the terminology coming from \cite{fs-oka}. Let $a\in\Pb^r$ and $x\in U.$ Since $\dim(L_a\cap X\cap Y)\leq l-1,$ there exists an $(k-l,l)$ Hartogs figure $H$ disjoint from $L_a\cap X\cap Y$ such that its hull $\widehat H$ is a neighborhood of $x$ in $U.$ By continuity of $a\mapsto L_a$, there exists a neighborhood $V$ of $a$ such that $L_{a'}\cap X\cap Y\cap H=\varnothing$ for all $a'\in V.$ Since $v\leq 0$ and $\ddc v=T_1-T_2$ where $T_2$ has continuous local potential, up to a continuous function $v$ is equal to a plurisubharmonic function on $\widehat H$ and \cite[Proposition 3.1]{fs-oka} implies that $vR_{a'}$ is well-defined for all $a'\in V.$ Moreover, again possibly by exchanging $v$ by $v+\phi$ with $\phi$ continuous, we can assume that $vR_{a'}\leq 0$ and $\ddc(vR_{a'})\geq0$ on $\widehat H.$ Hence, we can apply the Oka inequality \cite[Theorem 2.4]{fs-oka}. If $K$ is a compact set contained in the interior of $\widehat H$ then there exists a constant $C>0$  such that for all $a'\in V$
$$\|vR_{a'}\|_K\leq C\|vR_{a'}\|_H.$$
Since $a\mapsto R_a$ is continuous and $v$ is continuous on $\cup_{a'\in V}\supp(R_{a'})\cap H$, we obtain that the mass of $vR_{a'}$ is uniformly bounded on $K$ for $a'\in V$ and we conclude using the compactness of $\Pb^r.$

To prove the continuity of $a\mapsto vR_a,$ let $(a_n)_{n\geq1}$ be a sequence in $\Pb^r$ converging to $a_0.$ Since $vR_a$ has locally uniformly bounded mass, we can assume that $vR_{a_n}$ converges to a current $R'.$ We must have
$$\supp R'\subset\limsup_{n\to\infty}(\supp(R_{a_n}))\subset\limsup_{n\to\infty}(L_{a_n}\cap Y)\subset L_a\cap Y.$$
On the other hand, by continuity of $a\mapsto R_a$ and since $v$ is continuous on $U\setminus X,$ we have that $R'=vR_{a_0}$ outside $L_a\cap Y\cap X$ which has dimension smaller than or equal to $l-1.$ Hence,  the support theorem of Bassanelli \cite{bassanelli-cut-off} for currents $T$ such that $T$ and $\ddc T$ have order $0$ implies that $R'=vR_{a_0}$ on $U.$
\end{proof}
\begin{lemma}\label{le-int}
Let $(\nu_n)_{n\geq1}$ a sequence of probabilities in $\Pb^r$ which converges to $\nu.$ Let us define $R:=\int R_a d\nu$ and $R_n:=\int R_a d\nu_n.$ If $v R_a$ is well defined on $U$ and $a\mapsto vR_a$ is continuous then $vR_n$ and $vR$ are well-defined on $U$ and satisfy $vR_n=\int vR_ad\nu_n,$ $vR=\int vR_ad\nu$ and $\lim_{n\to\infty}vR_n=vR.$
\end{lemma}
\begin{proof}
Let $\phi$ be a $(l,l)$ smooth form with compact support in $U.$ We can assume that the support of $\phi$ is contained in a small ball $B \subset U$ on which $v=u_1-u_2$ where $u_1,$ $u_2$ are plurisubharmonic on $B$ and $u_2$ is continuous. Hence, there exists a decreasing sequence $(u_{1,j})_{j\geq1}$ of continuous plurisubharmonic functions on $B$ converging pointwise to $u_1.$ Define $v_j:=u_{1,j}-u_2.$ Since $vR_a$ is well defined then $\langle v_jR_a,\phi\rangle$ decreases to $\langle vR_a,\phi\rangle$ by \cite[Corollary 3.3]{fs-oka}. In particular, if we define $\psi_j(a):=\langle v_jR_a,\phi\rangle$ and $\psi(a):=\langle vR_a,\phi\rangle$ then $\psi_j$ decreases pointwise to $\psi$ which is a continuous function since $a\mapsto vR_a$ is continuous. Hence, by Dini's theorem $\psi_j$ converges uniformly to $\psi.$

On the other hand, by monotone convergence theorem $\langle v_jR,\phi\rangle$ decreases to $\langle vR,\phi\rangle$ which is potentially equal to $-\infty.$ But, by definition of $R$ we have
$$\lim_{j\to\infty}\langle v_jR,\phi\rangle=\lim_{j\to\infty}\int_{\Pb^r}\langle v_jR_a,\phi\rangle d\nu(a)=\lim_{j\to\infty}\int_{\Pb^r}\psi_jd\nu=\int_{\Pb^r}\psi d\nu=\int_{\Pb^r}\langle vR_a,\phi\rangle d\nu(a),$$
i.e. $vR=\int vR_ad\nu.$ The same holds for $R_n$, and then $\lim_{n\to\infty}vR_n=vR$ since $a\mapsto vR_a$ is continuous.
\end{proof}

\subsection{Pull-back of $(1,1)$-currents by rational maps}
In this subsection, $\pi\colon\Pb^k\dashrightarrow\Pb^r$ is a dominant rational map with $\dim(I(\pi))\leq q-1.$

As $\pi$ is not supposed to be a submersion on $\Pb^k\setminus I(\pi),$ the definition of the pull-back operator $\pi^*$ on currents requires some work. However, we will only consider currents given by wedge products of positive closed $(1,1)$-currents with continuous local potentials, which greatly simplifies the problem. If $\tau$ is a positive closed $(1,1)$-current on $\Pb^r$ which is equal locally to $\ddc u$ then $\pi_{|\Pb^k\setminus I(\pi)}^*\tau$ can be defined locally on $\Pb^k\setminus I(\pi)$ as $\ddc u\circ\pi.$ M\'eo \cite{meo} proved that $\tau\mapsto \pi_{|\Pb^k\setminus I(\pi)}^*\tau$ is continuous. Moreover, since $I(\pi)$ has codimension at least $2,$ the trivial extension of $\pi_{|\Pb^k\setminus I(\pi)}^*\tau$ to $\Pb^k$ is again  a positive closed $(1,1)$-current that we denote by $\pi^*\tau.$ We summarize in the following proposition the properties about pull-back we shall need in the sequel. Recall that if $R_1$ and $R_2$ are two positive closed currents on $\Pb^k$ of bidegree $(1,1)$ and $(j,j)$ respectively then the wedge product $R_1\wedge R_2$ is well-defined if the local potentials of $R_1$ are integrable with respect to $R_2\wedge\omega_{\Pb^k}^{k-j}$ (see e.g. \cite{bedford-taylor-capa}).
\begin{proposition}\label{prop-pull-back}
Let $\pi\colon\Pb^k\dashrightarrow\Pb^r$ be a dominant rational map whose indeterminacy set has a dimension smaller than or equal to $q-1.$ If $\tau$ is a positive closed $(1,1)$-current of mass $1$ on $\Pb^r$ then $\|\pi^*\tau\|$ is equal to the algebraic degree $\deg(\pi)$ of $\pi.$ If $\tau$ has continuous local potentials then the self-intersections $(\pi^*\tau)^j$ are well-defined for $j\in\{1,\ldots,r\}.$ The currents $(\pi^*\tau)^j$ coincide with the trivial extension of the standard pull-back of $\tau^j$ if $\tau$ is smooth. Moreover, if $(u_n)$ is a sequence of continuous functions which converges uniformly to $0$ then the currents $\tau_n:=\tau+\ddc u_n$ satisfy $\lim_{n\to\infty}(\pi^*\tau_n)^j=(\pi^*\tau)^j.$
\end{proposition}
\begin{proof}
Let $\tau$ be a positive closed $(1,1)$-current of mass $1$ on $\Pb^r.$ As we have said, in \cite{meo} M\'eo proved that  $\pi_{|\Pb^k\setminus I(\pi)}^*\tau$ depends continuously on $\tau.$ On the other hand, since $I(\pi)$ has codimension at least $2,$ the trivial extension, denoted by $\pi^*\tau,$ is a positive closed $(1,1)$-current on $\Pb^k.$ Moreover, as in the proof of Lemma \ref{le-conti}, the Oka inequality obtained in \cite{fs-oka} implies that the mass of $\pi^*\tau$ is bounded independently of $\tau.$ Hence, $\tau\mapsto\pi^*\tau$ is also continuous. Indeed, if $(\tau_n)_{n\geq1}$ is a sequence of positive closed $(1,1)$-currents converging to $\tau$ then $(\pi^*\tau_n)_{n\geq1}$ has uniformly bounded mass on $\Pb^k$ and using the continuity of $\pi_{|\Pb^k\setminus I(\pi)}^*$ each limit value $R$ has to be equal to $\pi_{|\Pb^k\setminus I(\pi)}^*\tau$ on $\Pb^k\setminus I(\pi).$ Finally, $R=\pi^*\tau$ since $I(\pi)$ has codimension at least $2.$ To see that $\|\pi^*\tau\|$ is in fact independent of $\tau,$ observe that if $\tau=\omega_{\Pb^r}+\ddc u$ where $u$ is a continuous function, then $u\circ\pi$ is in $L^1(\Pb^k)$ and $\langle\ddc(u\circ\pi),\omega_{\Pb^k}^{k-1}\rangle=0.$ Moreover, $(\pi^*\tau)-(\pi^*\omega_{\Pb^r})=\ddc(u\circ\pi)$ so $\|\pi^*\tau\|=\|\pi^*\omega_{\Pb^r}\|.$ The general case follows by continuity since smooth forms are dense in the space of positive closed $(1,1)$-currents on $\Pb ^r.$ Since $\|\pi^*\tau\|$ is independent of $\tau,$ we obtain that it is equal to $\deg(\pi)$ by taking an hyperplane in $\Pb^r.$

We now assume as in the statement that $\dim(I(\pi))\leq q-1.$ Let $R$ be a positive closed $(j,j)$-current on $\Pb^k$ with $j\in\{1,\ldots,r\}.$  If $\tau$ has continuous local potentials then $\pi^*\tau$ has continuous local potentials except on $I(\pi),$ i.e. the set of points where these local potentials are unbounded is contained in $I(\pi).$ Hence, using the assumption on $\dim(I(\pi))$ we can deduce from \cite{fs-oka} that these local potentials are integrable with respect to $R\wedge\omega_{\Pb^k}^{k-j}$ and thus $(\pi^*\tau)\wedge R$ is well-defined. In particular, $(\pi^*\tau)^j$ is well-defined for $j\in\{1,\ldots,r+1\}.$ The fact that these currents coincide with the trivial extension of $\pi_{|\Pb^k\setminus I(\pi)}^*(\tau^j)$ if $\tau$ is smooth and $j\in\{1,\ldots,r\}$ follows exactly as above. Observe however that for $j=r+1,$ $\pi^*(\tau^{r+1})$ vanishes whereas $(\pi^*\tau)^{r+1}$ has mass $\deg(\pi)^{r+1}$ and thus differs from $0.$ This implies that the support of $(\pi^*\tau)^{r+1}$ is contained in $I(\pi)$ and thus $\dim(I(\pi))\geq q-1$ i.e. $\dim(I(\pi))=q-1.$
 
We prove the last assertion by induction. The case $j=1$ follows from the first part of this proof. Assume the assertion is true for $j-1$ with $j\in\{2,\ldots,r\}.$ Observe that since $(u_n)$ converges uniformly to $0,$ $\Pb^k\setminus I(\pi)$ is covered by open sets $\Omega$ where we can write $\pi^* \tau=\ddc v$ and $\pi^* \tau_n=\ddc v_n$ where $(v_n)$ is a sequence of continuous functions converging uniformly to $v.$ Hence, if $\phi$ is a smooth form with compact support on $\Omega$ then
\begin{align*}
\langle(\pi^*\tau_n)^j-(\pi^*\tau)^j,\phi\rangle&=\langle\ddc(v_n(\pi^*\tau_n)^{j-1}-v(\pi^*\tau)^{j-1}),\phi\rangle
\\&=\langle v((\pi^*\tau_n)^{j-1}-(\pi^*\tau)^{j-1}),\ddc\phi\rangle+\langle(v_n-v)(\pi^*\tau_n)^{j-1},\ddc\phi\rangle.
\end{align*}
The inductive hypothesis implies that the first term in this sum converges to $0.$ The second term also converges to $0$ since $(v_n-v)$ converges uniformly to $0.$ Hence, any limit value of $(\pi^*\tau_n)^j$ has to be equal to $(\pi^*\tau)^j$ on $\Pb^k\setminus I(\pi)$ and this equality extends to $\Pb^k$ since $\dim(I(\pi))\leq q-1.$
\end{proof}
\begin{remark}
The degree $\deg(\pi)$ of $\pi$ is not necessary equal to $1.$ If $f$ preserves the fibration defined by $\pi$ then the fibration defined by $\pi\circ f$ is still preserved by $f$ and if $f$ has degree $d$ then $\deg(\pi\circ f)=d\deg(\pi).$ Another example in our setting is the binomial pencil of $\Pb^2$ considered in \cite{dabija-jonsson-pencil} and \cite{favre-pereira-foliation}  where the degree is an arbitrary integer.
\end{remark}
\begin{remark}
During the above proof, we have shown that $\dim(I(\pi))\geq q-1.$  This fact also follows from the definition of $I(\pi)$ as the common zeros of the $r+1$ homogeneous polynomials defining $\pi.$
\end{remark}

\section{Structure of the Green currents}\label{sec-struc}
This section is devoted to the proofs of Theorem \ref{th-mu}, Corollary \ref{cor-decomp} and Corollary \ref{cor-abso}. 

Let $f$ be an endomorphism of $\Pb^k$ of degree $d\geq2.$
Recall that the \textit{Green current} $T_f$ of $f$ can be defined as
$$T_f=\lim_{n\to\infty}\frac{1}{d^n}f^{n*}\omega_{\Pb^k}.$$
We refer to \cite{ds-lec} for a detailed study of this current. In what follows, we will use that $T_f$ has H\"{o}lder local potentials and if $l\in\{1,\ldots, k\}$ then its self-intersection $T_f^l$ satisfies $T_f^l=\lim_{n\to\infty}d^{-ln}f^{n*}\omega_{\Pb^k}^l.$

\begin{proof}[Proof of Theorem \ref{th-mu}]
First observe that since $I(\pi)\cap \Jc_q(f)=\varnothing,$ a cohomological arguments implies that the dimension of $I(\pi)$ is at most $q-1$ and $\pi$ satisfies the assumption of Proposition \ref{prop-pull-back}. Using this proposition, we define
$$R:=\deg(\pi)^{-1}\pi^*\omega_{\Pb^r}\ \ \text{ and } \ \ S:=\deg(\pi)^{-1}\pi^*T_\theta$$
which are two positive closed $(1,1)$-currents of mass $1.$ We also deduce from $I(\pi)\cap \Jc_q(f)=\varnothing$ that  there exists a neighborhood $U$ of $I(\pi)$ such that $\Jc_q(f)\cap U=\varnothing.$ Since $R$ is a smooth form on $\Pb^k\setminus I(\pi),$ there is a constant $C>0$ such that $R\leq C\omega_{\Pb^k}$ on $\Pb^k\setminus U$ and thus $T_f^q\wedge R^j\leq C^j T_f^q\wedge\omega_{\Pb^k}^j$  on $\Pb^k$ for $1\leq j\leq r.$ Applying the operator $d^{-n(q+j)}f^{n*}$ to this inequality gives
$$T_f^q\wedge\left(\frac{1}{d^{nj}}f^{n*}R^j\right)\leq C^j T_f^q\wedge\left(\frac{1}{d^{nj}}f^{n*}\omega_{\Pb^k}^j\right).$$
The equidistibution results for $f$ and the fact that $T_f$ has continuous local potentials (see \cite{ds-lec}) imply that the right-hand side converges to $C^jT_f^{q+j}.$ On the other hand, as $\pi\circ f^n=\theta^n\circ\pi$ we have by Proposition \ref{prop-pull-back}
$$T_f^q\wedge\left(\frac{1}{d^{nj}}f^{n*}R^j\right)=T_f^q\wedge\left(\frac{1}{\deg(\pi)^jd^{nj}}f^{n*}\pi^*\omega_{\Pb^r}^j\right)=T_f^q\wedge\pi^*\left(\frac{1}{\deg(\pi)^jd^{nj}}\theta^{n*}\omega_{\Pb^r}^j\right).$$
Recall that $d^{-n}\theta^{n*}\omega_{\Pb^r}=T_\theta+\ddc u_n$ where $u_n$ are continuous functions converging uniformly to $0.$ Hence, by Proposition \ref{prop-pull-back} the sequence above converges to $T_f^q\wedge S^j$ and thus $T_f^q\wedge S^j\leq C^j T_f^{q+j}.$ Moreover, $T_f^q\wedge S^j$ is invariant by $d^{-(q+j)}f^*$ and $T_f^{q+j}$ is extremal in the cone of such currents (see \cite{sibony-pano} when $q+j=1$ and \cite{ds-superpot} for the general case) so $T_f^q\wedge S^j=T_f^{q+j}.$

The proof of $S^j\neq T_f^j$ for $j\in\{1,\ldots,r\}$ simply comes from the fact (observed in the proof of Proposition \ref{prop-pull-back}) that $S^{r+1}$ is supported in $I(\pi)$ which has dimension at most $q-1.$ On the other hand, since $T_f$ has H\"{o}lder local potentials, it follows from \cite{sibony-pano} that $T_f^j\wedge S^{r+1-j}$ gives no mass to analytic sets of dimension $q-2+j\geq q-1.$

Finally, in order to prove the last assertion, observe that since $I(\pi)\cap \Jc_q(f)=\varnothing,$ the current $T_f^{q+j}=T_f^q\wedge S^j$ satisfies
\begin{equation}\label{eq-pi}
\pi_*(T_f^q\wedge S^j)=\deg(\pi)^{-j}\pi_*(T_f^q\wedge\pi^*T_\theta^j)=\deg(\pi)^{-j}(\pi_*T_f^q)\wedge T_\theta^j.
\end{equation}
The current $\pi_*T_f^q$ is positive and closed of bidegree $(0,0)$ on $\Pb^r$ thus it is a positive multiple of $[\Pb^r]$. Since $\pi_*$ preserves the mass of measures, the equation \eqref{eq-pi} with $j=r$ implies $\pi_*T_f^q=\deg(\pi)^r[\Pb^r]$ and thus  $\pi_*(T_f^{q+j})=\deg(\pi)^{r-j}T_\theta^j.$ In particular, $\pi_*\mu_f=\mu_\theta.$
\end{proof}

\begin{proof}[Proof of Corollary \ref{cor-decomp}]
We shall use the formula $\mu_f=T_f^q\wedge S^r$. The proof would have been straightforward if $\pi$ were a submersion on $\Pb^k\setminus I(\pi)$ and the current $T_f^q$ were smooth. However, since $T_f$ has continuous local potentials, we can use regularization as follows.

Let $\phi\colon\Pb^k\to\Rb$ be a continuous function. The critical set $\crit_\pi$ of $\pi$ is by definition the union of $I(\pi)$ with the set of points in $\Pb^k\setminus I(\pi)$ where the differential of $\pi$ has rank strictly less than $r.$ This set is algebraic so the measure $\mu_f$ gives no mass to it. Hence, if $\chi_n\colon\Pb^k\to[0,1]$ are smooth functions  with compact support in $\Pb^k\setminus\crit_\pi$ which converge locally uniformly to $1$ on $\Pb^k\setminus\crit_\pi$ then $\langle\mu_f,\phi\rangle=\lim_{n\to\infty}\langle\mu_f,\chi_n\phi\rangle.$ Moreover, for $\mu_\theta$-almost all $a\in\Pb^r$ the algebraic set $L_a:=\overline{\pi^{-1}_{|\Pb^k\setminus I(\pi)}(a)}$ has dimension $q$ and $L_a\cap\crit_\pi$ has dimension $q-1.$ In particular, the current $[L_a]$ has mass $\deg(\pi)^r$ and coincides with the trivial extension of $\pi_{|\Pb^k\setminus\crit_\pi}^*\delta_a.$ This implies that if $\Psi$ is a smooth $(q,q)$-form on $\Pb^k$ then  the function $\pi_*(\chi_n\phi\Psi)$ on $\Pb^r$ is equal $\mu_\theta$-everywhere to $a\mapsto\langle\Psi\wedge[L_a],\chi_n\phi\rangle.$

On the other hand, the current $T_f$ has continuous local potentials so there exists a continuous function $g$ such that $T_f=\omega_{\Pb^k}+\ddc g.$ Let $(g_l)_{l\geq1}$ be a sequence of smooth functions converging uniformly to $g$ and define $T_l:=\omega_{\Pb^k}+\ddc g_l.$ The uniform convergence implies that if $R$ is a positive closed $(r,r)$-current then $T^q_l\wedge R$ converges to $T_f^q\wedge R.$ Hence, $\mu_f=T^q_f\wedge S^r$ implies
\begin{align*}
\langle\mu_f,\phi\rangle&=\lim_{n\to\infty}\langle\mu_f,\chi_n\phi\rangle=\lim_{n\to\infty}\lim_{l\to\infty}\langle T_l^q\wedge S^r,\chi_n\phi\rangle=\lim_{n\to\infty}\lim_{l\to\infty}\langle\mu_\theta,\pi_*(\chi_n\phi T_l^q)\rangle/\deg(\pi)^r\\
&=\lim_{n\to\infty}\lim_{l\to\infty}\int_{\Pb^r}\langle T^q_l\wedge\frac{[L_a]}{\deg(\pi)^r},\chi_n\phi\rangle d\mu_\theta(a)=\lim_{n\to\infty}\int_{\Pb^r}\langle T^q_f\wedge\frac{[L_a]}{\deg(\pi)^r},\chi_n\phi\rangle d\mu_\theta(a)\\
&=\int_{\Pb^r}\langle T^q_f\wedge\frac{[L_a]}{\deg(\pi)^r},\phi\rangle d\mu_\theta(a),
\end{align*}
where the last equality comes from the fact that for $\mu_\theta$-almost all $a,$ $L_a\cap\crit_\pi$ has dimension $q-1$ and $T^q_f\wedge[L_a]$ gives no mass to such sets.
\end{proof}

By the Radon-Nikodym theorem, the following result implies Corollary \ref{cor-abso}.
\begin{corollary}
Under the assumptions of Theorem \ref{th-mu}, if there exists a positive $\omega_{\Pb^r}^r$-integrable function $h$ such that $\mu_\theta=h\omega^r_{\Pb^r}$ then there exists $A>0$ such that $\mu_f\leq A \, (h\circ\pi) \, \sigma_{T_f^q}.$ In particular, $\mu_f<<\sigma_{T_f^q}.$
\end{corollary}

\begin{proof}
Let $R:=\deg(\pi)^{-1}\pi^*\omega_{\Pb^r}$ as in the proof of Theorem \ref{th-mu}. The same theorem gives that $\mu_f=\deg(\pi)^{-r}T_f^q\wedge(\pi^*\mu_\theta)=\frac{h\circ\pi}{\deg(\pi)^r} T_f^q\wedge R^r.$ On the other hand, we have seen in the proof of Theorem \ref{th-mu} that $T_f^q\wedge R^r\leq C^r\sigma_{T_f^q}$ thus $\mu_f\leq\left(\frac{C}{\deg(\pi)}\right)^r \, (h \circ\pi) \, \sigma_{T_f^q}.$
\end{proof}

\section{Lyapunov exponents}\label{sec-lyap}
This section is dedicated to the proof of Theorem \ref{th-lyap}. The main ingredients are the formula $\pi_*\mu_f=\mu_\theta,$ the fact that these measures put no mass on proper analytic sets and the local uniform convergence in Oseledec Theorem. We refer the reader to \cite[Chapter 5--6]{barreira-pesin-ergodic} for the details on Oseledec theorem we shall need.
\begin{proof}[Proof of Theorem \ref{th-lyap}]
First, let us set some notations. Let $\crit_\theta$ be the critical set of $\theta.$ We denote by $\widehat \Pb^r:=\{(y_n)_{n\in\Zb}\in \Pb^r\,|\, \theta(y_n)=y_{n+1}\}$ the natural extension and by $\widehat\theta$ the left-shift on $\widehat \Pb^r.$ The projection $\proj\colon\widehat \Pb^r\to \Pb^r$ defined by $\proj((y_n)_{n\in\Zb})=y_0$ satisfies $\theta\circ \proj=\proj\circ\widehat\theta.$ The measure $\mu_\theta$ has a unique lift $\widehat\mu_\theta$ which is invariant by $\widehat\theta$ and such that $\proj_{\theta*}\widehat\mu_\theta=\mu_\theta.$ In what follows, if $\widehat y$ is in $\widehat\Pb^r$ we will write $y_0$ instead of $\proj(\widehat y).$
The measure $\widehat\mu_\theta$ inherits several properties from $\mu_\theta.$ Since $\mu_\theta$ is ergodic and integrates the quasi-plurisubharmonic functions, it follows that $\widehat\mu_\theta$ is ergodic, gives no mass to $\cup_{p \in \Zb} \hat \theta^p (\proj^{-1} (\crit_\theta)) $ and integrates the functions $\log\|D\theta^{\pm1}\|.$ In particular, for $\widehat\mu_\theta$-almost all $\widehat y=(y_n)_{n\in\Zb}$ the differentials
$$D_{\widehat y}\theta^{-n}:=(D_{y_{-n}}\theta^n)^{-1},\ \ \ D_{\widehat y}\theta^n:=D_{y_0}\theta^n$$
are well defined for $n\geq1$. Moreover, by Oseledec theorem, there exist distincts numbers $\Lambda_1,\ldots,\Lambda_s\in\Rb$ and a $\widehat\theta$-invariant set $Y\subset\widehat \Pb^r,$ included in $(\Pb^r \setminus \crit_\theta)^\Zb,$ such that $\widehat\mu_\theta(Y)=1$ and for each $\widehat y\in Y$ the tangent space $T_{y_0}\Pb^r$ admits a splitting $T_{y_0}\Pb^r=\bigoplus_{i=1}^sV_i(\widehat y)$ which satisfies
$$D_{y_0}\theta(V_i(\widehat y))=V_i(\widehat\theta(\widehat y))\ \text{ and }\ \lim_{n\to\pm\infty}\frac{1}{n}\log\|D_{\widehat y}\theta^n(u)\|=\Lambda_i$$
uniformly for $u\in V_i(\widehat y)$ with $\|u\|=1.$ By definition, the multiplicity of $\Lambda_i$ is the dimension $m_i$ of $V_i(\widehat y).$ Moreover, the subspaces $V_i(\widehat y)$ can be characterized by $V_i(\widehat y)\setminus\{0\}=\{u\in T_{y_0}\Pb^r\setminus\{0\}\,|\, \lim_{n\to\pm\infty} n ^{-1}\log\|D_{\widehat y}\theta^n(u)\|=\Lambda_i\}.$

The natural extension $\widehat\Pb^k,$ the map $\widehat f$ and the measure $\widehat\mu_f$ are defined in the same way with respect to $f$ and since $\pi\circ f=\theta\circ\pi,$ the map $\pi$ lifts to a map $\widehat\pi\colon\widehat\Pb^k\to\widehat\Pb^r$ such that $\widehat\pi\circ\widehat f=\widehat\theta\circ\widehat\pi.$ The uniqueness of the lift $\widehat\mu_\theta$ of $\mu_\theta$ and the fact that $\pi_*\mu_f=\mu_\theta$ imply that $\widehat\pi_*\widehat\mu_f=\widehat\mu_\theta.$ In particular, the set $\widehat\pi^{-1}(Y)$ has full $\widehat\mu_f$-measure. The measure $\widehat\mu_f$ also admits a Oseledec decomposition on a $\widehat f$-invariant set $Z$ of full $\widehat\mu_f$-measure. And since the measure $\mu_f$ gives no mass to proper analytic sets, if $\crit_\pi$ denotes the critical set of $\pi$ (i.e. the union of $I(\pi)$ with the set of points in $\Pb^k\setminus I(\pi)$ where the differential of $\pi$ has rank strictly less than $r$) then  the $\widehat f$-invariant set
$$X:=Z\cap\widehat\pi^{-1}(Y)\cap(\Pb^k\setminus(\crit_f\cup\crit_\pi))^\Zb$$
also has full $\widehat\mu_f$-measure.

Now, fix $i\in\{1,\ldots,s\}$ and for $\widehat x\in X$ consider the subspace of $T_{x_0}\Pb^k$ defined by
$$W_i(\widehat x)=(D_{x_0}\pi)^{-1}(V_i(\widehat\pi(\widehat x))).$$
As $X$ is disjoint from the critical sets of $f$ and $\pi$, these subspaces have dimension $m_i+q$ and define a $Df$-invariant distribution i.e. $D_{x_0}f(W_i(\widehat x))=W_i(\widehat f(\widehat x)).$ Therefore, Oseledec theorem applies to the measurable cocycle defined by the action of $Df$ on $W_i(\widehat x)$ and induces $\widehat\mu_f$-almost everywhere a decomposition  $W_i(\widehat x)=\bigoplus_{j=1}^lF_{ij}(\widehat x)$ for some integer $l\geq1.$ Since this cocycle is a sub-cocycle of the standard one, each $F_{ij}(\widehat x)$ is associated to a Lyapunov exponent $\lambda_j$ of $\mu_f$ and satisfies 
$$ F_{ij}(\widehat x)\setminus\{0\}=\{v\in W_i(\widehat x)\setminus\{0\}\,|\, \lim_{n\to\pm\infty}n ^{-1}\log\|D_{\widehat x}f^n(v)\|=\lambda_j\}. $$ 
In particular, $\dim(F_{ij}(\widehat x))$ is bounded by the multiplicity of $\lambda_j.$
Now we show that if $F_{ij}(\widehat x)$ is not contained in $\ker D_{x_0}\pi$ then $\lambda_j=\Lambda_i$, that property will be sufficient to conclude. So let $\widehat x\in X$ and $j\in\{1,\ldots,l\}$ such that $F_{ij}(\widehat x)\not\subset\ker D_{x_0}\pi.$ In particular, there exists $v\in F_{ij}(\widehat x)\setminus\ker D_{x_0}\pi$ and so
\begin{align*}
\Lambda_i=\lim_{n\to + \infty}\frac{1}{n}\log\|D_{\widehat x}(\theta^n\circ\pi)(v)\|&=\lim_{n\to + \infty}\frac{1}{n}\log\|D_{\widehat x}(\pi\circ f^n)(v)\|\\
&\leq\lim_{n\to + \infty}\frac{1}{n}\log\|D_{\widehat x}f^n(v)\|=\lambda_j.
\end{align*}
Here, the inequality comes from $I(\pi)\cap\supp(\mu_f)=\varnothing$ and thus, there exists $C>0$ such that $\|D_x\pi(v)\|\leq C\|v\|$ for all $x\in\supp(\mu_f)$ and $v\in T_x\Pb^k.$

For the converse inequality, observe that since the subspaces $W_i(\widehat x)$ and $F_{ij}(\widehat x)$ depends measurably on $\widehat x,$ there exist a constant $c>0$ and $B\subset X$ with $\widehat\mu_f(B)>0$ such that for each $\widehat x\in B$ there is a subspace $E(\widehat x)\subset F_{ij}(\widehat x)$ of positive dimension satisfying $E(\hat x) \cap \ker D_{x_0}\pi = \{ 0 \}$ and $\|D_{x_0}\pi(v)\|\geq c\|v\|$ for all $v\in E(\widehat x)$. By Poincar\'e recurrence theorem, for $\widehat\mu_f$-almost all $\widehat x\in X$ there exists an increasing sequence $(k_n)_{n \geq 1}$ of integers such that  $\widehat f^{k_n}(\widehat x)\in B$ for all $n \geq 1$. Therefore, if $v_n\in F_{ij}(\widehat x)$ is such that $\|v_n\|=1$ and $D_{\widehat x}f^{k_n}(v_n)\in E(\widehat f^{k_n}(\widehat x))$ then
\begin{align*}
\lambda_j=\lim_{n\to + \infty}\frac{1}{k_n}\log\|D_{\widehat x}f^{k_n}(v_n)\|&\leq\lim_{n\to + \infty}\frac{1}{k_n}\log\|D_{\widehat x}(\pi\circ f^{k_n})(v_n)\|\\
&=\lim_{n \to + \infty}\frac{1}{k_n}\log\|D_{\widehat x}(\theta^{k_n}\circ\pi)(v_n)\|\leq\Lambda_i.
\end{align*}
Here, the first equality comes from the fact that the convergence in Oseledec theorem is uniform on the unit sphere of $F_{ij}(\widehat x).$ The last inequality uses a similar argument as well as the uniform bound $\|D_{x}\pi(v)\|\leq C\|v\|$ for $x\in\supp(\mu_f)$ and the fact that $v_n\notin\ker D_{x_0}\pi$ since  $D_{\widehat x}f^{k_n}(v_n)\in E(\widehat f^{k_n}(\widehat x)).$

We have shown that if $F_{ij}(\widehat x)\not\subset\ker D_{x_0}\pi$ then $\lambda_j=\Lambda_i$. Thus there is a unique $j \in \{ 1 , \ldots , l \}$ with this property. As $\dim(W_i(\widehat x))=m_i+q$ and $\dim(\ker D_{x_0}\pi)=q,$ the dimension of $F_{ij}(\widehat x)$ has to be at least $m_i.$
\end{proof}

\section{Generalized Bedford-Jonsson's formula}\label{sec-bj}
In this section, we assume that the fibration $\pi$ is the standard linear fibration $\pi\colon\Pb^k\dashrightarrow\Pb^r$ defined by $\pi[y:z]=[y]$ where $y:=(y_0,\ldots,y_r)\in\Cb^{r+1}$ and $z=(z_0,\ldots,z_{q-1})\in\Cb^q.$ We first analyse some basic properties of maps preserving such a fibration and then we give a proof of Theorem \ref{th-bj}.

As we have said in the introduction, the indeterminacy set of $\pi$ corresponds to $\{y=0\}\simeq \Pb^{q-1}$ and each fiber $L_a:=\overline{\pi^{-1}(a)}$ is a projective space $\Pb^q$ in which $I(\pi)$ can be identified with the hyperplane at infinity, i.e. $L_a\setminus I(\pi)\simeq \Cb^q.$ If $f\colon\Pb^k\to\Pb^k$ preserves such a fibration then it lifts to a polynomial endomorphism $F$ of $\Cb^{k+1}$ of the form
$$F(y,z)=(\Theta(y),R(y,z)),$$
where $y\in\Cb^{r+1},$ $z\in\Cb^q$ and $\Theta,$ $R$ are homogeneous polynomials. The map $\Theta$ is a lift of the endomorphism $\theta$ of $\Pb^r$ such that $\theta\circ\pi=\pi\circ f.$ The inequality $\|\Theta(y)\|_\infty\leq\|F(y,z)\|_\infty$ implies that the functions
$$G_\Theta(y):=\lim_{n\to\infty}\frac{1}{d^n}\log\|\Theta^n(y)\|_\infty\ \ \text{ and } \ \ G_F(y,z):=\lim_{n\to\infty}\frac{1}{d^n}\log\|F^n(y,z)\|_\infty,$$
satisfy $G_\Theta(y)\leq G_F(y,z).$ The difference of these functions goes down to $\Pb^k$ and we define it as the \textit{relative Green function} of $f,$ $G[y:z]:=G_F(y,z)-G_\Theta(y).$ It is the unique lower semicontinuous function such that
$$\ddc G=T_f-S\ \ \text{ and }\ \ \min G=0.$$
Here, $T_f$ is the Green current of $f$ and $S:=\pi^*T_\theta.$
It is easy to check that $G$ is invariant (i.e. $d^{-1}G\circ f=G$), continuous on $\Pb^k\setminus I(\pi)$ and $\{G=+\infty\}=I(\pi).$ As we will see in the proof of the next lemma, $I(\pi)$ is an attracting set for $f$ and $G$ encodes the speed of convergence toward it. In particular, the assumptions of Theorem \ref{th-mu} are satisfied.
\begin{lemma}\label{le-supp-sans-para}
If $f$ preserves the linear fibration $\pi\colon\Pb^k\dashrightarrow\Pb^r$, then $\Jc_q(f)\cap I(\pi)=\varnothing.$
\end{lemma}
\begin{proof}
Using the notations introduced above, we will prove that if $\epsilon>0$ is small enough then the region
$$U_\epsilon:=\{[y:z]\in\Pb^k\,|\,\|y\|_\infty<\epsilon\|z\|_\infty\}$$
satisfies $f(\overline{U_\epsilon}) \subset \overline{U_{\epsilon/2}}$. In particular $\cap_{n\geq1}f^n(U_\epsilon)=I(\pi).$ Let us define
$$\alpha:=\max_{\|y\|_\infty=1}\|\Theta(y)\|_\infty\ \ \text{ and }\ \ \beta:=\min_{\|z\|_\infty=1}\|R(0,z)\|_\infty.$$
Since $f$ is a well-defined endomorphism of $\Pb^k,$ we have $\beta>0.$ Thus, if $[y:z]\in \overline{U_\epsilon}$ then
$$\|\Theta(y)\|_\infty\leq\alpha\|y\|^d_\infty\leq\alpha\epsilon^d\|z\|_\infty^d\ \text{ and }\ \|R(y,z)\|_\infty\geq \beta\|z\|_\infty^d-\gamma\epsilon\|z\|_\infty^d,$$
where $\gamma$ is the sum of the moduli of the coefficients of $R(y,z)-R(0,z).$ Therefore, if $\epsilon>0$ is small enough and $[y:z]\in\overline{U_\epsilon}$ then
$$\|\Theta(y)\|_\infty\leq\alpha\epsilon^d\|z\|_\infty\leq\epsilon(\beta-\gamma\epsilon)\|z\|^d_\infty/2\leq\epsilon\|R(y,z)\|_\infty/2,$$
which gives $f[y:z]\in \overline{U_{\epsilon/2}}$.

On the other hand, if $H$ is a generic linear subspace of $\Pb^k$ of dimension $r$ then $H\cap I(\pi)=\varnothing$ thus $H\cap U_\epsilon=\varnothing$ for $\epsilon>0$ small enough. Hence, we can regularize the current of integration $[H]$ to obtain a positive closed $(q,q)$ smooth form $\widetilde\omega$ of mass $1$ supported in $\Pb^k\setminus U_\epsilon.$ Equidistribution results for smooth forms (see \cite{ds-lec}) give $T_f^q=\lim_{n\to\infty}d^{-nq}f^{n*}\widetilde\omega.$ The fact that $f(U_\epsilon)\subset U_\epsilon$ implies that $\supp(T_f^q)\cap U_\epsilon=\varnothing$ and thus $\Jc_q(f)\cap I(\pi)=\varnothing.$
\end{proof}
\begin{remark}
The fact that $I(\pi)$ is an attracting set of dimension $q-1$ implies directly by \cite{t-attractor} that $\Jc_q(f)\cap I(\pi)=\varnothing.$ However, we will use the same proof as above in a parametric setting (see Lemma \ref{le-supp}).
\end{remark}
The next result follows easily from the fact that $f$ preserves the fibration defined by $\pi.$
\begin{lemma}\label{le-crit}
The critical current of $f$ admits a decomposition
$$[\crit_f]=[C_\infty]+[C_\sigma],$$
where $[C_\infty]=\pi^*[\crit_\theta]$ and $[C_\sigma]$ is the current of integration on an algebraic set, called the sectional part of $\crit_f.$
\end{lemma}
\begin{proof}
If $\rho\colon\Cb^{k+1}\setminus\{0\}\to\Pb^k$ is the standard projection then the critical current of $f$ can be defined by $\rho^*[\crit_f]=\ddc\log|\det DF|.$
The fact that $F$ has the form $F(y,z)=(\Theta(y),R(y,z))$ implies that $\det DF=\det D\Theta\times\det D_zR$ where $D_zR$ denotes the $q\times q$ matrix formed by the partial derivatives of $R$ in the $z_0,\ldots,z_{q-1}$ directions. Hence,
$$[\crit_f]=[C_\infty]+[C_\sigma]$$
where $\rho^*[C_\sigma]=\ddc\log|\det D_zR|$ and $\rho^*[C_\infty]=\ddc\log|\det D\Theta|.$ It is easy to check that $[C_\infty]=\pi^*[\crit_\theta].$
\end{proof}

If $a$ is a $n$-periodic point of $\theta$ then $f^n_{|L_a}$ can be identified to a regular polynomial endomorphism of $\Cb^q\simeq L_a\setminus I(\pi)$ as follows. Since $\theta^n(a)=a,$ there exists $A=(a_0,\ldots,a_r)\in\Cb^{r+1}$ such that $a=[a_0:\cdots:a_r]$ and $\Theta^n(A)=A.$ Hence, if $[y:z]\in\Pb^k$ belongs to $L_a\setminus I(\pi)$ then there exists a unique $Z\in\Cb^q$ such that $[y:z]=[A:Z]$ and $f^n_{|L_a\setminus I(\pi)}$ can be identified with $R_n(Z):=R_{\Theta^{n-1}(A)}\circ\cdots\circ R_A(Z)$ where $R_A(Z):=R(A,Z).$ With the same notations, $G[y:z]=G_F(A,Z)-G_\Theta(A)$ and, as $\Theta^n(A)=A$ implies $G_\Theta(A)=0,$ we have
$$G[y:z]=\lim_{l\to\infty}\frac{1}{d^{ln}}\max(\log\|R_n^l(Z)\|_\infty,0),$$
which is exactly the Green function associated to the polynomial mapping $R_n.$ Hence, using this identification the equilibrium measure of $f^n_{|L_a}$ is $T_f^q\wedge[L_a].$ The critical current of $f^n_{|L_a}$ is $[\crit_{f^n}]\wedge[L_a]$ if $a$ is not a critical point of $\theta^n$ (i.e. if the wedge product is well-defined). In fact, without any assumption on $a,$ since $[\crit_{f^n}]=\sum_{i=0}^{n-1}f^{i*}[\crit_f]$ one can check using Lemma \ref{le-crit} that the critical current associated to the restriction of $f^n$ to $L_a \simeq \Pb^q$ corresponds to $\left(\sum_{i=0}^{n-1}f^{i*}[C_\sigma]\right)\wedge[L_a]+d^n[I(\pi)].$ Hence, the Bedford-Jonsson formula for regular polynomial endomorphisms of $\Cb^q$ (see Theorem \ref{bdjn}) yields the following result.
\begin{lemma}\label{le-lyap}
Let $a\in\Pb^r$ be such that $\theta^n(a)=a.$ If $\Lambda_0$ (resp. $\Lambda(f^n_{|L_a})$) denotes the sum of the Lyapunov exponents of $f_{|I(\pi)}$ (resp. $f^n_{|L_a}$) with respect to its equilibrium measure then
$$\Lambda(f^n_{|L_a})=n\log d+n\Lambda_0+\sum_{i=0}^{n-1}\langle T_f^{q-1}\wedge[C_\sigma]\wedge[L_{\theta^i(a)}],G\rangle.$$
\end{lemma}
\begin{proof}
Since the dynamics of $R_n$ on the hyperplane at infinity can be identified to the one of $f^n$ on $I(\pi),$ the discussion above and the Bedford-Jonsson formula give 
$$\Lambda(f^n_{|L_a})=n\log d+n\Lambda_0+\left\langle T_f^{q-1}\wedge\left(\sum_{i=0}^{n-1}f^{i*}[C_\sigma]\right)\wedge[L_a],G\right\rangle.$$
We conclude by using 
$$f^i_*(GT^{q-1}_f\wedge[L_a])=GT^{q-1}_f\wedge\frac{f^i_*[L_a]}{d^{qi}}\ \ \text{ and } \ \ \frac{f^i_*[L_a]}{d^{qi}}=[L_{\theta^i(a)}], $$
where the first equality follows from the invariance of $T_f$ and $G$, and the second one from the invariance of the fibration. 
\end{proof}
Observe that since $\pi$ is a submersion on $\Pb^k\setminus I(\pi),$ we have $$S^r =(\pi^*T_\theta)^r=\pi^*\mu_\theta=\int_{\Pb^r}[L_a]d\mu_\theta(a) $$
on $\Pb^k\setminus I(\pi)$. 
 As $I(\pi)$ has dimension $q-1,$ these equalities extends to $\Pb^k.$
This allows us to use the continuity results obtained in Section \ref{sec-basics} to prove the following result.
\begin{lemma}\label{le-lyap-approx}
$$\Lambda_\sigma=\lim_{n\to\infty}\frac{1}{nd^{rn}}\sum_{\theta^n(a)=a}\Lambda(f^n_{|L_a}).$$ 
\end{lemma}
\begin{proof}
By Lemma \ref{le-supp-sans-para}, there exists a neighborhood $\Omega$ of $I(\pi)$ such that $U:=\Pb^k\setminus\overline{\Omega}$ contains $\Jc_q(f).$
Since $f$ preserves the fibration defined by $\pi,$ its differential preserves the subbundle $\ker D\pi$ of the tangent bundle over $\Pb^k\setminus I(\pi).$ Hence, we can define on $U$ the Jacobian $|\Jac_\sigma f|$ of $Df$ in the direction of $\ker D\pi$ with respect to a smooth metric. The function $u:=\log|\Jac_\sigma f|$ is bounded from above on $U$ and is locally the sum of a potential on $[C_\sigma]$ and a smooth function. In particular, it satisfies $\ddc u=[C_\sigma]+(T_1-T_2)$ on $U$, where $T_1$ and $T_2$ are two positive smooth forms. Since $\Lambda_\sigma = \Lambda_f - \Lambda_\theta$ (see the remark after Theorem \ref{bdjn}) and $\pi_* \mu_f = \mu_\theta$, we have $\Lambda_\sigma=\langle\mu_f,u\rangle.$ On the other hand, $u$ is continuous on $U\setminus C_\sigma$ and $\dim(L_a\cap C_\sigma)=q-1$. Hence,  by Lemma \ref{le-conti} with $v=u$, $X=C_\sigma$ and $Y=\Pb^k,$ we have that for each $a\in\Pb^r$ the current $u[L_a]$ is well defined on $U$ and depends continuously on $a.$ Moreover, since $T_f$ has continuous local potentials, $a\mapsto u[L_a]\wedge T_f^q$ is also continuous. Therefore, using that $\mu_f=T_f^q\wedge S^r=\int T_f^q\wedge [L_a]d\mu_\theta$, we obtain by Lemma \ref{le-int} that
\begin{equation}\label{eq-lyap-1}
\Lambda_\sigma=\int_{\Pb^r}\langle T_f^q\wedge[L_a],u\rangle d\mu_\theta=\lim_{n\to\infty}\frac{1}{d^{rn}}\sum_{\theta^n(a)=a}\langle T_f^q\wedge[L_a],u\rangle,
\end{equation}
where the last equality comes from the equidistribution of periodic points of $\theta$ towards $\mu_\theta$ (see \cite{briend-duval-expo}). Let us recall that the measure $T_f^q\wedge[L_a]$ corresponds to the equilibrium measure of the polynomial mapping $f^n_{|L_a}.$ Therefore, if $\Lambda(f^n_{|L_a})$ denotes the sum of the Lyapunov exponents of $f^n_{|L_a}$ with respect to this measure, we have by definition of $u$
\begin{align}\label{eq-lyap-2}
\Lambda(f^n_{|L_a})=\langle T_f^q\wedge[L_a],\log|\Jac f^n_{|L_a}|\rangle&=\sum_{i=0}^{n-1}\langle T_f^q\wedge[L_a], u\circ f^i\rangle=\sum_{i=0}^{n-1}\langle f^i_*(T_f^q\wedge[L_a]),u\rangle\nonumber\\
&=\sum_{i=0}^{n-1}\langle T_f^q\wedge [L_{\theta^i(a)}],u\rangle,
\end{align}
where the last equality comes from $d^{-q}f^*T_f^q=T_f^q$ and $d^{-q}f_*[L_a]=[L_{\theta(a)}].$ Combining \eqref{eq-lyap-1} and \eqref{eq-lyap-2} gives $\Lambda_\sigma=\lim_{n\to\infty}\frac{1}{nd^{rn}}\sum_{\theta^n(a)=a}\Lambda(f^n_{|L_a}).$
\end{proof}

We can now finish the proof of Theorem \ref{th-bj}.
\begin{proof}[Proof of Theorem \ref{th-bj}]
Lemma \ref{le-lyap} and Lemma \ref{le-lyap-approx} imply
\begin{align*}
\Lambda_\sigma&=\lim_{n\to\infty}\frac{1}{nd^{rn}}\sum_{\theta^n(a)=a}\left(n\log d+n\Lambda_0(f)+\sum_{i=0}^{n-1}\langle T_f^{q-1}\wedge[C_\sigma]\wedge[L_{\theta^i(a)}],G\rangle\right)\\
&=\log d+\Lambda_0(f)+\lim_{n\to\infty}\frac{1}{d^{rn}}\sum_{\theta^n(a)=a}\langle T^{q-1}_f\wedge[C_\sigma]\wedge[L_a],G\rangle\\
&=\log d+\Lambda_0(f)+\langle T_f^{q-1}\wedge S^r\wedge[C_\sigma],G\rangle,
\end{align*}
where the last equality comes from Lemma \ref{le-conti} and Lemma \ref{le-int} applied with $v=-G,$ $U=\Pb^k,$ $X=I(\pi)$ and $Y=[C_\sigma].$ To be more precise, as we have seen before $a\mapsto u[L_a]$ is continuous and $\ddc u=[C_\sigma]+(T_1-T_2)$ on $U$ where $T_1$ and $T_2$ are smooth. Hence, $a\mapsto[C_\sigma]\wedge[L_a]$ is continuous. Since $\dim(I(\pi)\cap C_\sigma)=q-2,$ Lemma \ref{le-conti} implies that $a\mapsto G[C_\sigma]\wedge[L_a]$ is continuous. Thus, the continuity of the local potentials of $T_f$ gives that $a\mapsto G[C_\sigma]\wedge[L_a]\wedge T_f^{q-1}$ is continuous. And finally, Lemma \ref{le-int} gives $\langle T^{q-1}\wedge S^r\wedge[C_\sigma],G\rangle=\int\langle T^{q-1}\wedge [L_a]\wedge[C_\sigma],G\rangle d\mu_\theta=\lim_{n\to\infty}\int\langle T^{q-1}\wedge [L_a]\wedge[C_\sigma],G\rangle d\mu_n,$ where $\mu_n$ is the average of the Dirac masses on the $n$-periodic points of $\theta.$
\end{proof}

\section{Sectional and fiber-wise bifurcation currents}\label{sec-bif}

In this section, we consider a family $(f_\lambda)_{\lambda\in M}$ of endomorphisms of $\Pb^k$ which preserves the standard linear fibration i.e. there exists a family $(\theta_\lambda)_{\lambda\in M}$ of endomorphisms of $\Pb^r$ such that $\pi\circ f_\lambda=\theta_\lambda\circ\pi$, where $\pi$ is defined in Section \ref{sec-bj}. We denote by $\Pi\colon M\times\Pb^k\dashrightarrow M\times\Pb^r,$ $f\colon M\times\Pb^k\to M\times\Pb^k$ the maps $\Pi(\lambda, x)=(\lambda,\pi(x)),$ $f(\lambda,x)=(\lambda,f_\lambda(x))$ and by $p_{\Pb^k}\colon M\times\Pb^k\to\Pb^k,$ $p\colon M\times\Pb^k\to M$ the two projections.

It is possible to define a $(1,1)$-current in $M\times\Pb^k$ whose slices by $p$ are exactly the Green currents $T_{f_\lambda}$ of $f_\lambda$, we refer to \cite{ds-geo} for a definition of the slices in our situation. Indeed, by \cite{bassanelli-berteloot-bif}
the following limit exists
$$T_f:=\lim_{n\to\infty}\frac{1}{d^n}f^{n*}(p_{\Pb^k}^*\omega_{\Pb^k}),$$
and defines the Green current associated to the family $(f_\lambda)_{\lambda\in M}.$ The approach of \cite{bassanelli-berteloot-bif} consists to lift (locally) the family $(f_\lambda)$ to a family $(F_\lambda)$ of polynomial endomorphisms of $\Cb^{k+1}$ and show that potentials of the lift of $d^{-n}f^{n*}(p_{\Pb^k}^*\omega_{\Pb^k})$ converge locally uniformly to a function, called the Green function of the family $(F_\lambda),$ which by definition is a potential of the lift of $T_f.$ Moreover, since this convergence is locally uniform, $T_f$ has continuous local potentials, its self-intersections $T_f^l$ are well-defined and satisfy $T_f^l=\lim_{n\to\infty}d^{-nl}f^{n*}(p_{\Pb^k}^*\omega_{\Pb^k}^l)$ for $l\in\{1,\ldots,k\}.$ When $l=k,$ this gives a $(k,k)$-current $T^k_f$ on $M\times\Pb^k$ whose slices are the equilibrium measure of $f_\lambda.$ A positive closed $(k,k)$-current with this property is called an \textit{equilibrium current} in \cite{pham}. Bassanelli-Berteloot proved that the bifurcation current $T_\bif(f):=\ddc\Lambda_f$ satisfies
$$T_\bif(f)=p_*(T_f^k\wedge[\crit_f]),$$
and Pham obtained this result in the more general setting of polynomial-like maps and proved that $T^k_f$ can be replaced by any equilibrium current, see \cite[pages 8-9]{pham}.

Following \cite{astorg-bianchi-skew} where the special case of polynomial skew products has been studied, we are interested by the relationship between the bifurcation currents $T_\bif(f)$ associated to $(f_\lambda)_{\lambda\in M}$ and $T_\bif(\theta)$ associated to $(\theta_\lambda)_{\lambda\in M}$ when $\theta_\lambda\circ\pi=\pi\circ f_\lambda.$ There are two reasons to restrict ourselves to the linear fibration. The first one is that $\pi$ is a submersion on $\Pb^k\setminus I(\pi).$ It implies in particular that the critical sets $\crit_{f_\lambda}$ have a decomposition into a "sectional" part and a "fibered" part, the latter being given by $\pi^{-1}(\crit_{\theta_\lambda})$ (see Lemma \ref{le-sub}). The second reason is that the support of $T_f^q$, which is the Green current of order $q$ of the family $(f_\lambda)_{\lambda\in M}$, is disjoint from $M\times I(\pi).$ Indeed, it is easy to check that for all $\lambda\in M$ the map $f_\lambda$ satisfies the condition $\Jc_q(f_\lambda)\cap I(\pi)=\varnothing.$ However, for an arbitrary family it is not clear whether this condition for each parameter implies $\supp(T_f^q)\cap M\times I(\pi)=\varnothing.$
\begin{lemma}\label{le-supp}
Let $(f_\lambda)_{\lambda\in M}$ and $\pi$ be as above. Then $\supp(T_f^q)\cap M\times I(\pi)=\varnothing$ and $\Pi_*(T_f^q)=[M\times\Pb^r].$
\end{lemma}
\begin{proof}
We will use the same arguments as for Lemma \ref{le-supp-sans-para} locally in the parameter space. Let $\lambda_0$ be a fixed parameter in $M.$ Since the region $U_\epsilon=\{[y:z]\in\Pb^k\,|\,\|y\|_\infty<\epsilon\|z\|_\infty\}$ is a trapping region for $f_{\lambda_0}$ for $\epsilon>0$ small enough, then if $N$ is a small enough neighborhood of $\lambda_0$ then $f(N\times U_\epsilon)\subset N\times U_\epsilon.$ Let $\widetilde\omega$ denote the positive closed $(q,q)$-form supported in $\Pb^k\setminus U_\epsilon$ obtained in the proof of Lemma \ref{le-supp-sans-para}. There exists a smooth $(q-1,q-1)$-form $\phi$ on $\Pb^k$ such that $\widetilde\omega=\omega_{\Pb^k}^q+\ddc\phi$ and there exists $C>0$ such that $-C\omega_{\Pb^k}^{q-1}\leq\phi\leq C\omega^{q-1}_{\Pb^k}.$ Let $p_{\Pb^k}\colon N\times\Pb^k\to\Pb^k$ be the canonical projection. As we have said in the beginning of this Section, 
$$\lim_{n\to\infty}\frac{1}{d^{n(q-1)}}f^{n*}(p_{\Pb^k}^*\omega_{\Pb^k}^{q-1})=T^{q-1}_f.$$
This and the inequalities $-C\omega_{\Pb^k}^{q-1}\leq\phi\leq C\omega^{q-1}_{\Pb^k}$ imply that $\lim_{n\to\infty}\frac{1}{d^{nq}}f^{n*}(p_{\Pb^k}^*\phi)=0$ and thus
$$\lim_{n\to\infty}\frac{1}{d^{nq}}f^{n*}(p_{\Pb^k}^*\widetilde\omega)=\lim_{n\to\infty}\frac{1}{d^{nq}}f^{n*}(p_{\Pb^k}^*\omega_{\Pb^k}^q)=T^q_f.$$
Therefore, since $\supp(p^*_{\Pb^k}\widetilde\omega)\subset N\times(\Pb^k\setminus U_\epsilon)$ and $f(N\times U_\epsilon)\subset N\times U_\epsilon$, we have  $\supp(T_f^q)\cap N\times I(\pi)=\varnothing$  on $N\times\Pb^k$, as desired. This implies that the current $\Pi_*(T_f^q)$ is a well-defined positive closed $(0,0)$-current on $M\times\Pb^r,$ i.e. it is equal to $\alpha[M\times\Pb^r]$ for some $\alpha>0.$ Finally $\alpha = 1$ since the fibers of $\Pi$ have degree $1.$
\end{proof}
The next result follows exactly as Lemma \ref{le-crit}.
\begin{lemma}\label{le-sub}
Let $(f_\lambda)_{\lambda\in M},$ $(\theta_\lambda)_{\lambda \in M}$ and $\pi$ be as in Theorem \ref{th-pham}. Then the current of integration on the critical set $\crit_f$ of the family $(f_\lambda)_{\lambda\in M}$ has a decomposition
$$[\crit_f]=[C_\infty]+[C_\sigma],$$
where $[C_\infty]=\Pi^*[\crit_\theta]$ and $[C_\sigma]$ is the current of integration on an analytic set.
\end{lemma}

\begin{remark}
The fact that the critical set is not irreducible for the family $(f_\lambda)_{\lambda\in M}$ implies directly that the bifurcation current $T_\bif(f)$ admits a similar decomposition. For a general family in one variable (i.e. $k=1$), possibly by exchanging $M$ by a branched cover, each critical point can be followed individually thus $\crit_f$ has as many irreducible components as there are critical points. This gives rise to the decomposition of $T_\bif(f)$ into the currents associated to the activation of each critical point (see \cite{dujardin-favre-bif}).
\end{remark}

The last ingredient to prove Theorem \ref{th-pham} is the following result about slicing.
\begin{lemma}\label{le-equilibrium}
Let $(f_\lambda)_{\lambda\in M},$ $(\theta_\lambda)_{\lambda_\in M}$ and $\pi$ be as in Theorem \ref{th-pham}. If $S:=\Pi^*(T_\theta)$ then $T^q_f\wedge S^r$ is an equilibrium current for $(f_\lambda)_{\lambda\in M}.$
\end{lemma}
\begin{proof}
It follows easily from the definition that if $R$ is a positive closed current in $M\times\Pb^k$ such that the slice $R_\lambda$ is well-defined and if $u$ is a continuous function on $\supp(R)$ then $(uR)_\lambda=u_{\{\lambda\}\times\Pb^k}R_\lambda$ and $(\ddc uR)_\lambda=\ddc (u_{\{\lambda\}\times\Pb^k}R_\lambda).$ This implies that the slice $(T_f^q)_\lambda$ of $T^q_f$ is equal to $((T_f)_\lambda)^q$ which is the Green current of order $q$ of $f_\lambda.$ Moreover,  Lemma \ref{le-supp} yields $\supp(T_f^q)\cap M\times I(\pi)=\varnothing,$  therefore $S=\Pi^*(T_\theta)$ has continuous local potentials on $\supp(T_f^q)$. Thus the slice $(T^q_f\wedge S^r)_\lambda$ is equal to $((T_f)_\lambda)^q\wedge(\pi^*(T_\theta)_\lambda)^r$ which is equal to the equilibrium measure of $f_\lambda$ by Theorem \ref{th-mu},  i.e. $T^q_f\wedge S^r$ is an equilibrium current for $(f_\lambda)_{\lambda\in M}.$
\end{proof}

\begin{proof}[Proof of Theorem \ref{th-pham}]
We denote by $p\colon M\times\Pb^k\to M$ and $\widetilde p\colon M\times\Pb^r\to M$ the two projections. Observe that $p=\widetilde p\circ\Pi$ on $M\times(\Pb^k\setminus I(\pi)).$ Since $T^q_f\wedge S^r$ is an equilibrium current, it follows from Pham's article \cite{pham} that
$$\ddc\Lambda_f=p_*(T^q_f\wedge S^r\wedge[\crit_f])=p_*(T^q_f\wedge S^r\wedge[C_\infty])+p_*(T^q_f\wedge S^r\wedge[C_\sigma]),$$
where the last inequality comes from the decomposition obtained in Lemma \ref{le-sub}. Moreover, as $\Lambda_f=\Lambda_\theta+\Lambda_\sigma,$ in order to prove that $\ddc\Lambda_\sigma=p_*(T^q_f\wedge S^r\wedge[C_\sigma])$ it is sufficient to prove that $\ddc\Lambda_\theta=p_*(T^q_f\wedge S^r\wedge[C_\infty]).$ To this end, observe that by Lemma \ref{le-supp}
$$p_*(T^q_f\wedge S^r\wedge[C_\infty])=\widetilde p_*(\Pi_*(T^q_f\wedge \Pi^*(T_\theta^r\wedge[\crit_\theta])))=\widetilde p_*(T_\theta^r\wedge[\crit_\theta]).$$
On the other hand, Bassanelli-Berteloot formula applied to the family $(\theta_\lambda)_{\lambda\in M}$ gives $\ddc\Lambda_\theta=\widetilde p_*(T_\theta^r\wedge[\crit_\theta])$ which concludes the proof.
\end{proof}
Using this formula of $T_{\bif,\sigma}(f)$, we can rely this current to the bifurcations in the fibers.

\begin{proof}[Proof of Corollary \ref{cor-bif}]
Define
$$\Lambda_{\sigma,n}(\lambda):=\frac{1}{nd^{rn}}\sum_{\theta_\lambda^n(a)=a}\Lambda(f^n_{\lambda|L_a}).$$
By Lemma \ref{le-lyap-approx} for each $\lambda\in M$ we have
$$\Lambda_\sigma(\lambda)=\lim_{n\to\infty}\Lambda_{\sigma,n}(\lambda).$$
Moreover, by the Briend-Duval inequality on Lyapunov exponents, the function $\Lambda_{\sigma,n}(\lambda)$ is uniformly bounded from below by $(q\log d)/2.$ It is also locally uniformly bounded from above since $\Lambda_{\sigma,n}(\lambda)\leq q\log(\max_{x\in\Pb^k}(\|D_xf_\lambda\|)).$ On the other hand,  Theorem \ref{th-pham} implies that $\ddc\Lambda_\sigma=p_*(T_f^q\wedge S^r\wedge[C_\sigma]).$ Therefore, if $[\per_{\theta,n}]$ denotes the current of integration on $\{(\lambda,a)\in M\times\Pb^r\,|\, \theta^n_\lambda(a)=a\}$ (where we take into account the multiplicities) and we set $S_n:=\Pi^*[\per_{\theta,n}]/d^{rn},$ in order to prove the corollary, it is sufficient to prove that $\ddc\Lambda_{\sigma,n}=p_*(T_f^q\wedge S_n\wedge[C_\sigma])$ for each $n\geq1.$ 
Indeed, this ensures that $\Lambda_{\sigma,n}$ is plurisubharmonic, hence it converges to $\Lambda_\sigma$ in $L^1_{loc}$ (unique cluster value), hence $T_{\bif,n}(f)$ converges to $T_{\bif,\sigma}(f)$.

To this end, let $n\geq1$ and observe that outside a ramification locus $\Sigma_n\subset M$ of codimension at least $1$ the periodic points of $\theta_\lambda$ of period $n$ can be followed holomorphically i.e. each $\lambda\in M\setminus\Sigma_n$ admits a neighborhood $N$ and a family $\{\gamma_j\}_{j\in J}$ of holomorphic maps $\gamma_j\colon N\to\Pb^r$ such that $[\per_{\theta,n}]=\sum_{j\in J}[\Gamma_j]$ on $N\times\Pb^r$ where $\Gamma_j$ is the graph of $\gamma_j.$ If $j\in J$ and $\Theta(\lambda,y):=(\lambda,\theta_\lambda(y))$ then $\Theta(\Gamma_j)$ is another graph $\Gamma_{j'}$ with $j'\in J.$ Hence, if we gather the periodic points of a same cycle, we obtain a current
$$\sum_{i=0}^{n-1}p_*(T_f^q\wedge\Pi^*[\Theta^i(\Gamma_j)]\wedge[C_\sigma])$$
which is equal to the bifurcation current of the family $(f^n_{\lambda|L_{\gamma_j(\lambda)}})_{\lambda\in N}$ of endomorphisms of $L_{\gamma_j(\lambda)} \simeq \Pb^q$, i.e. equal to $\ddc\Lambda(f^n_{\lambda|L_{\gamma_j(\lambda)}}).$ Since this equality is true for every $j \in J$, we get the formula $\ddc\Lambda_{\sigma,n}=p_*(T_f^q\wedge S_n\wedge[C_\sigma])$ on $M\setminus\Sigma_n.$ To conclude, observe that none of these currents gives mass to $\Sigma_n$. Indeed, as we have remarked before $\Lambda_{\sigma,n}$ is uniformly bounded from below by $(q\log d)/2$ so $\ddc\Lambda_{\sigma,n}$ gives no mass to the proper analytic set $\Sigma_n$. For the second current, the analytic set $X_n:=\supp(S_n\wedge[C_\sigma])$ has dimension $m+q-1,$ where $m:=\dim(M)$, and its intersections with the fibers of $p$ have dimension $q-1.$ Hence, the dimension of $X_n\cap p^{-1}(\Sigma_n)$ is strictly less than $m+q-2$ and the current $T_f^q\wedge S_n\wedge[C_\sigma]$ gives no mass to this set since $T_f$ has continuous local potentials. This implies that $p_*(T_f^q\wedge S_n\wedge[C_\sigma])$ gives no mass to $\Sigma_n$, as desired.
\end{proof}

\newcommand{\etalchar}[1]{$^{#1}$}

\noindent {\footnotesize Christophe Dupont}\\
{\footnotesize Univ Rennes}\\
{\footnotesize CNRS, IRMAR - UMR 6625}\\
{\footnotesize F-35000 Rennes, France}\\
{\footnotesize christophe.dupont@univ-rennes1.fr}\\

\noindent {\footnotesize Johan Taflin}\\
{\footnotesize Universit\'e de Bourgogne Franche-Comt\'e}\\
{\footnotesize IMB, CNRS UMR 5584}\\
{\footnotesize 21078 Dijon Cedex, France}\\
{\footnotesize johan.taflin@u-bourgogne.fr}\\

\end{document}